\documentclass[12pt,a4paper]{amsart}
\usepackage[utf8]{inputenc}
\usepackage{enumerate}   
\usepackage{amsmath}
\usepackage{amsfonts}
\usepackage{float}
\usepackage{amssymb}
\usepackage{graphicx}
\usepackage{enumerate}
\usepackage{mathtools}
\usepackage{url}
\usepackage[left=2cm,right=2cm,top=2cm,bottom=2cm]{geometry}
\newcommand{\orb}[1]{O_{#1}} 
\newcommand{\countequal}[1]{c_{#1}}
\newcommand{\countupto}[1]{C_{#1}}

\DeclareMathOperator{\wl}{\mathrm{WL}}
\DeclareMathOperator{\gl}{\mathrm{GL}}
\DeclareMathOperator{\si}{\mathrm{SIN}}

\def\map#1#2#3{\mbox{${#1}\colon {#2} \longrightarrow {#3}$}}

\theoremstyle{plain}
\newtheorem{theorem}{Theorem}[section]

\newtheorem{proposition}[theorem]{Proposition}

\newtheorem{conjecture}{Conjecture}
\newtheorem{open problem}{Open Problem}

\theoremstyle{definition}

\theoremstyle{remark}
\newtheorem{remark}[theorem]{Remark}


\makeindex

\title[Non-abelian number theory]{Non-Abelian number theory and the structure of curves on surfaces}
\author{Moira Chas}
\begin{document}
\maketitle
\begin{abstract} In this note we study numerically the combinatorics of curves and geodesics on the torus with one boundary component. A potential computational difficulty is avoided by counting inside specific orbits of the mapping class group up to a certain length, either geometric or combinatorial. Some cases are rigurolosly determined and the Euler totient function emerges. More complicated orbits are computed to suggest an array of formulae continuing to involve the Euler totient function.  We formulate six precise conjectures. These include a novel study of an ``inverse'' function of the Mirzakhani's asymptotics. 
The geometric part of our study was motivated by the rationality aspect of these asymptotics. 

\end{abstract}

\section{Introduction}
The patterns discovered here are based on computer experiments using algorithms about curves on surfaces with given intersection properties. 
Even though the number of classes of curves up to a certain length has exponential growth, the curves with given self-intersection properties are much thinner say of polynomial growth. For example, Mirzakhani's thesis \cite{mir} gives polynomial estimates for embedded curves in general surfaces. We make  use here  of a useful astuce  in that work to consider orbits of the component group $M(g,n)$ of  homeomorphisms of the surface $S$ of genus $g$ and $n$ punctures  to overcome the numerical difficulty presented by the exponential search for thin subsets with fixed intersection number. 

Our computations begin with the torus with one boundary component. We try to understand how many curves there are in one orbit of $M(1,1)$ with fixed intersection number $k$, for $k=0,1,2,3$ and word length $\ell$ (relative to a given set of generators $a, b$ of the fundamental group). We count unoriented curves to simplify the matter, since all curves oriented counts  double the unoriented counts.

A very elementary observation is that for self-intersection number $0$ there is only one interesting orbit $\alpha$ and one has an exact formula $2\Phi(\ell)$ for the number of elements of word length $\ell$ in its orbit, denoted $\countequal{\ell}(\alpha)$, where $\Phi(\ell)$ is the number of positive integers less than $\ell$ which are relatively prime to $\ell$ (Euler's totient function). The sum over $n$ up  to $\ell$ of $2\Phi(n)$ is asymptotic to $\frac{6}{\pi^2}\ell^2$ \cite{hw}, thus the number of elements in the orbit of $\alpha$ up to word length $\ell$ is asymptotic to $\frac{6}{\pi^2}\ell^2$.

In the experimental work we found these patterns hold for  more complicated orbits of more general words. Namely, for $k=1$
and partially for $k=2$ we can rigorously compute the cardinality of the orbits of curves of self-intersection $k$ of word length $\ell$ in terms of the Euler totient function $\Phi$ evaluated at various points, see Table~\ref{t2} . This precise result implies that the limit as $\ell$ tends to infinity of the sum up to $\ell$ is a certain rational number times $\frac{6}{\pi^2}$ as coefficients of $\ell^2$ for the asymptotics. See Tables~\ref{t2}, \ref{t3}, Conjecture~\ref{palpha}.

The first steps of these results were presented with related experiments in \cite{c15} and later discussed in a lecture  the Sixth Iberoamerican Congress on Geometry, CUNY, May 2015.

 Word  came from two sources who either heard the CUNY lecture or read the notes \cite{c15} (Erlandsson-Souto \cite{es}, Mirzakhani \cite{m16}) that the asymptotics $c L^2$,  I was alluding to, could be rigorously proven in the context of geometric length,  instead of our context of word length. Furthermore, I learned from Mirzakhani, that her constant factored into a piece only depending on the orbit, a piece only depending on the topology of the surface, and a piece only depending on the geometry. It was now incumbent to run experiments with geometric length in order to compare those constants with our findings for word length.

To  do this one  must compute geometrically. We describe here how  to construct explicit geometries on the torus with one geodesic boundary component and how to compute \emph{all} the terms in an orbit up to a specific geometric length. \emph{The ensuing computations showed the Mirzakhani constants depending on $\alpha$ were indeed rational and actually agreed with the rationals we had found in the combinatorial word length setting. }

Finally, we propose that these patterns that intertwine asymptotics of geometry and asymptotics of combinatorial group theory with the primes via Euler's totient function for the torus with one boundary component extend to general surfaces of genus $g$  and $n$ boundary components. All in all the subject, both  experimental and theoretical, takes the shape of a kind of non-Abelian Number theory. This of course fits well with the 70's discovery of Bill Thurston that at the level of moduli spaces the integral projective structure on the boundary of the hyperbolic plane (tantamount to the p/q linear foliations of the torus) generalizes to a piecewise integral projective structure on Thurston's boundary of the higher genus marked moduli space.

\textbf{ Acknowledgements: } The geometric construction of the metric on the torus with one boundary component depended on conversations with Bernard Maskit followed up by a discussion with Feng Luo.
\section{Preliminaries}

We study  unoriented, non-power  free homotopy classes of closed curves on   the torus with one puncture or one boundary component.
Choose a pair $(a,b)$ of generators of the fundamental group of the one holed torus. (This choice is equivalent to choosing a pair of disjoint simple arcs with limiting endpoints at the puncture.)

Denote the inverses of $a$ and $b$ by $A$ and $B$ respectively (see Figure~\ref{gen}).

For a free homotopy class $\alpha$, $\si(\alpha)$ denotes the self-intersection number (that is, the smallest number of intersection points of representatives of $\alpha$ counted with multiplicity). The number of letters of the shortest word representing $\alpha$  in the $\{a,b,A,B\}$-alphabet is denoted by $\wl(\alpha)$. 
 One needs to think of these words as cyclically reduced rings of letters, (even though for typographical reasons they are written as linear words) because we consider free homotopy classes. (Recall that the fundamental group of the punctured torus is the free group in two generators, and that free homotopy classes of oriented curves are in bijective correspondence with conjugacy classes of the fundamental group). 

Define
$$
\orb{}(\alpha)=\{\beta  \mbox{ in the mapping class group orbit of $\alpha$} \},
$$
$$
\orb{\ell}(\alpha)=\{\beta  \mbox{ in $\orb{}(\alpha)$ such that } \wl(\beta)=\ell \},
$$
$$
\countequal{\ell}(\alpha)=\# \orb{\ell}(\alpha) \mbox{ and }
\countupto{\ell}(\alpha)=\sum_{n=1}^\ell \countequal{n}(\alpha)
$$
Intuitively, $\orb{}(\alpha)$ labels the ``topological picture'' of $\alpha$ on the surface, up to homeomorphism.	
 
We will  present representatives of certain orbits to explain and  illustrate our results. The generators we use are depicted in Figure~\ref{gen}. A cyclic word labeling a free homotopy class is obtained by recording  the crossing of the ``a'' and ``b'' arcs, with $a$ (resp. $b$) if the direction of the arrows is respected and $A$ (resp. $B$) otherwise.

\begin{figure}
\includegraphics[scale=.2]{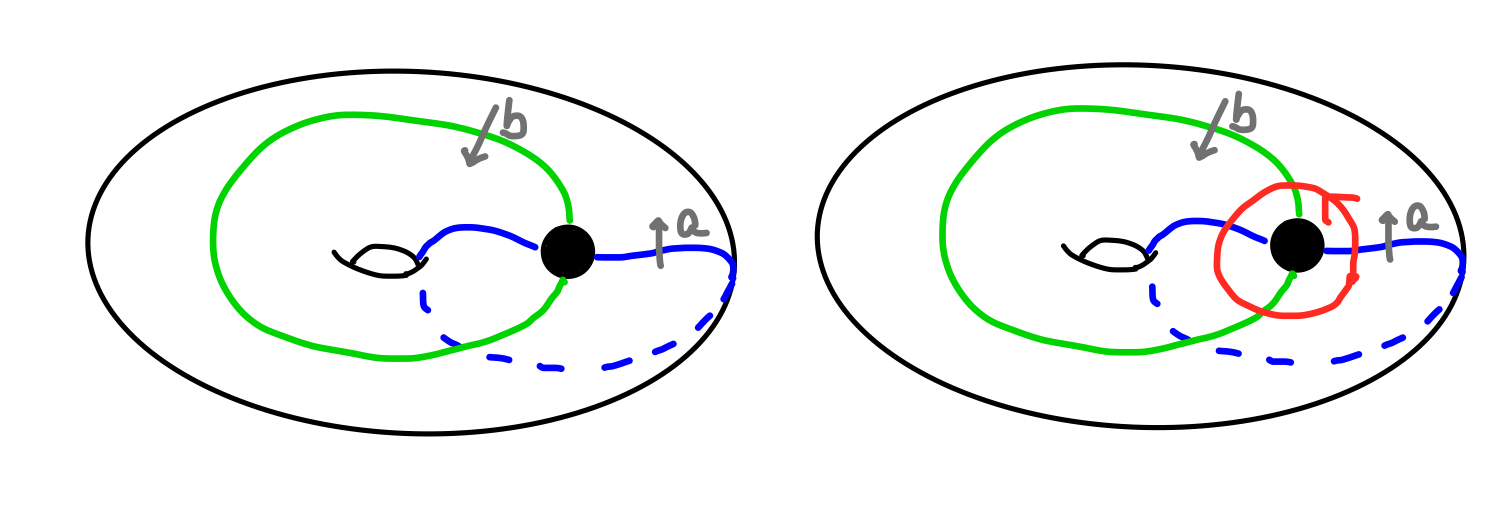} 
\caption{Arcs determining generators of the fundamental group of the punctured torus (left) and a representative in red of the class $abAB$ (right)}\label{gen}
\end{figure}

\section{Growth of the number of  classes of self-intersection number either zero, one, two or three as a function of the word length}

Consider one of the chosen generators of the fundamental group, say $a$. Observe that $a$ is simple (i.e. represented by an embedded closed curve). In order to simplify notation, we will denote the free homotopy class containing $a$ by $a$ also. 
Because such simple  curves do not separate, one sees that   classes in $\orb{\ell}(a)$ are in one to one correspondence with ordered pairs of  relatively prime integers $(j,k)$ with $k \ge 0$ and such that $|j|+|k|=\ell$. 
 Thus, we have the following proposition,  
 where $\Phi$ denotes the Euler totient function.
($\Phi(\ell)$ is the number of positive integers smaller than $\ell$ and relatively prime with $\ell$. )
 
\begin{proposition} \label{simple} The number of elements of word length $\ell$ in the orbit of a non-self-intersecting curve, not parallel to the boundary component is $2\Phi(\ell)$. Therefore, the number of elements of word length less than or    equal  to $\ell$ grows like $\frac{6}{\pi^2}\ell^2$. In symbols, 

 $\countequal{\ell}(a)=2\Phi(\ell)$ and 
$\countupto{\ell}(a)=\sum_{n=1}^{\ell}  2\Phi(n)$ which grows like $\frac{6}{\pi^2}\ell^2$ \cite{hw}.

\end{proposition}

There are two orbits of simple curves, the orbit of $a$ and the orbit of $abAB$, the orbit of the curve that goes around  the punctured (which is fixed by the mapping class group). Now, we discuss orbits of self-intersection number one, see Figure \ref{one1}.

\begin{proposition}\label{selfint one}
\begin{enumerate}
\item The number of elements of word length $\ell$ in the  orbit of $aabAB$ is $4\Phi(\ell-4)$
\item The number of elements of word length $\ell$ in the orbit of $abaB$ is $2\Phi(\ell/2)$ if $\ell$ is even and $0$ otherwise.
\end{enumerate}
\end{proposition}
\begin{proof}
We define a map  $\map{f}{ \orb{\ell}(aabAB) }{\orb{\ell-4}(a)}$.
If $\alpha \in \orb{\ell}(aabAB)$, set $f(\alpha)$ as the free homotopy class obtained by "cutting off" the loop of $\alpha$ that goes around the boundary component. Since $\alpha$ has self-intersection number one, the loop of $\alpha$ around the boundary component intersects each of the four arcs emanating (these are the arcs associated to the chosen generators) from the boundary component once. Thus, removing the loop about the boundary components  is the same as removing four letters ($abAB$ or inverse) for the word of $\alpha$. This implies that $\map{f}{ \orb{\ell}(aabAB) }{\orb{\ell-4}(a)}$. 

To see that the map $f$ is is two-to-one and surjective, take a simple curve $\beta$ and represent it by a closed geodesic $g$ on the flat torus that misses the puncture. Consider a short segment through the puncture at right angles to the inclination of $g$, extended both ways to hit $g$. At each of these points one can add a loop to $g$ around the puncture to create a self-intersection.

One can  study the orbit of $abaB$ in a  way similar to the way for  the orbit of $aabAB$: to obtain a map $$\map{g}{ \orb{\ell}(abaB) }{\mbox{\{square of simple classes in closed torus\}}}$$ given by filling in the boundary component.  This map is one to one.

\end{proof}

\begin{proposition}
 The number of classes of self-intersection number  one up to word length  $\ell$ grows like  $\countupto{\ell}(aabAB)+\countupto{\ell}(abaB)$,  $(2+\frac{1}{4})\frac{6}{\pi^2}\ell^2=\frac{27}{2\pi^2}\ell^2$.

\end{proposition}

\begin{figure}
\includegraphics[scale=.2]{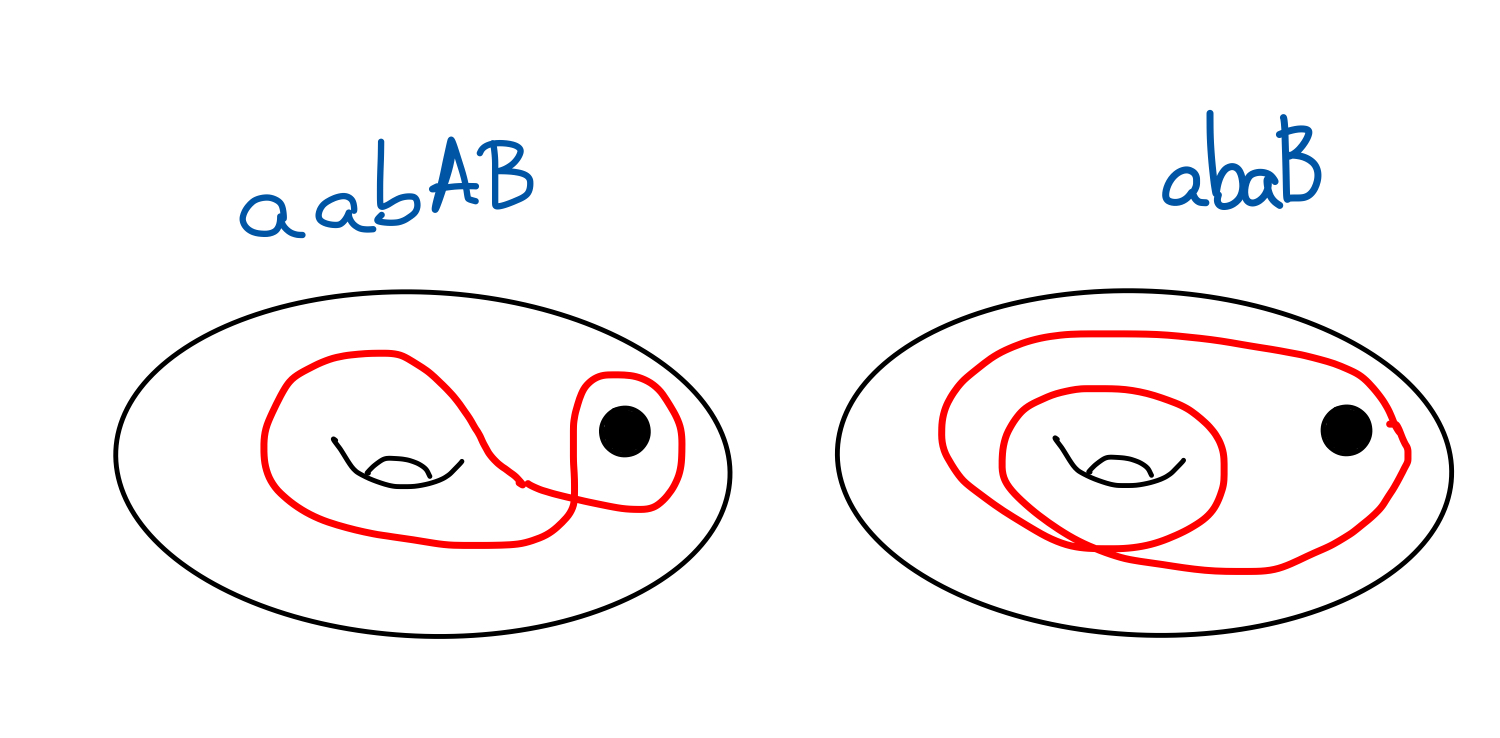} 
\caption{A representative of each orbit of self-intersection number one}\label{one1}
\end{figure}

From computer experiments we find then a conjecture about the number of elements of a given word length in each of the orbits of elements of self-intersection one, two, and three. 

There are six   orbits of curves of self-intersection two \cite{crispetal}, see Figure \ref{two2}. The REU preprint \cite{bck} claims that there are fifteen orbits of curves of self-intersection three, but according to our computer experiments, confirmed by work of our undergraduate student Joseph Suk, there are  fourteen. (Self-intersection three orbits are not pictured).

\begin{remark} The proofs of the counting of the number of elements in the orbits of $a^nabAB$ and  $a(abAB)^n$ are very similar to that for $aabAB$.
\end{remark}

{\renewcommand{\arraystretch}{1.5}
\begin{table}
\begin{tabular}{|c|c|c|c|c|c|c|}
\hline
&$\alpha$ & $p_\alpha$ & $\si(\alpha)$ & $\wl$ & $\countequal{\ell}(\alpha)$&Special cases\\
\hline
\hline
known& \text{a} & 1& 0 & All &$2\Phi(\ell)$&-\\
 \hline
 \hline
Prop.\ref{selfint one} & \text{aabAB} & 2  &1& All&$4\Phi(\ell-4)$, if $\ell\ge 5$&-\\
 \hline
Prop.\ref{selfint one} & \text{abaB} & $\frac{1}{4}$  &1& Even &$2\Phi(\ell/2)$, if $\ell\ge 6$&$\countequal{4}(\alpha)=4$\\
 \hline\hline
like Prop.\ref{selfint one} & \text{a(abAB)$^2$} & 2 &2& All&$4\Phi(\ell-8)$, if $\ell\ge 10$&$\countequal{9}(\alpha)=4$\\
 \hline
Unproven& \text{aaabb} & 2&2 & All&$2\Phi(\ell)+2\Phi(\ell-4)$, if $\ell\ge 6$&$\countequal{5}(\alpha)=8$\\
 \hline
 
Unproven &\text{aabAAB} & $\frac{1}{4}$ &2& Even &$2\Phi(\ell/2-2)$, if $\ell\ge 5$&-\\

 \hline 
like Prop.\ref{selfint one}  & \text{a$^2$abAB} & $\frac{1}{2}$ &2&Even&$4\Phi(\ell/2-2)$, if $\ell\ge 5$&-\\
 \hline

Unproven & \text{abaBabAB} & $\frac{1}{2}$ &2& Even &$4\Phi(\ell/2-2)$, if $\ell\ge 9$&$\countequal{8}(\alpha)=8$\\
 \hline
like Prop.\ref{selfint one}  & \text{aabaB} & $\frac{2}{9}$ &2& Mult. of   3 &$4\Phi(\ell/3)$, if $\ell\ge 6$&$\countequal{5}(\alpha)=4$\\
 \hline

 \hline
\end{tabular} 
\caption{The number of elements of word length $\ell$ in the orbit of curves of self-intersection number one and two. If nothing is said about a word length then the number of elements is zero. For instance,  $\countequal{3}(abaB)=0$). The number of elements up to word length $\ell$ in the orbit of $\alpha$, $\countupto{\ell}(\alpha)$ grows like $\frac{6}{\pi^2}p_\alpha$, where $p_\alpha$ is a rational number that depends on $\alpha$ (second column). The coefficients $p_\alpha$ are obtained using the fact that   $\sum_{n=1}^{\ell}  \Phi(n)$  grows like $\frac{3}{\pi^2}\ell^2$.
The formulae labeled unproven holds up to very long word length, even for the orbits of self-intersection three below.
}\label{t2}
\end{table}

}

\begin{figure}
\includegraphics[scale=.2]{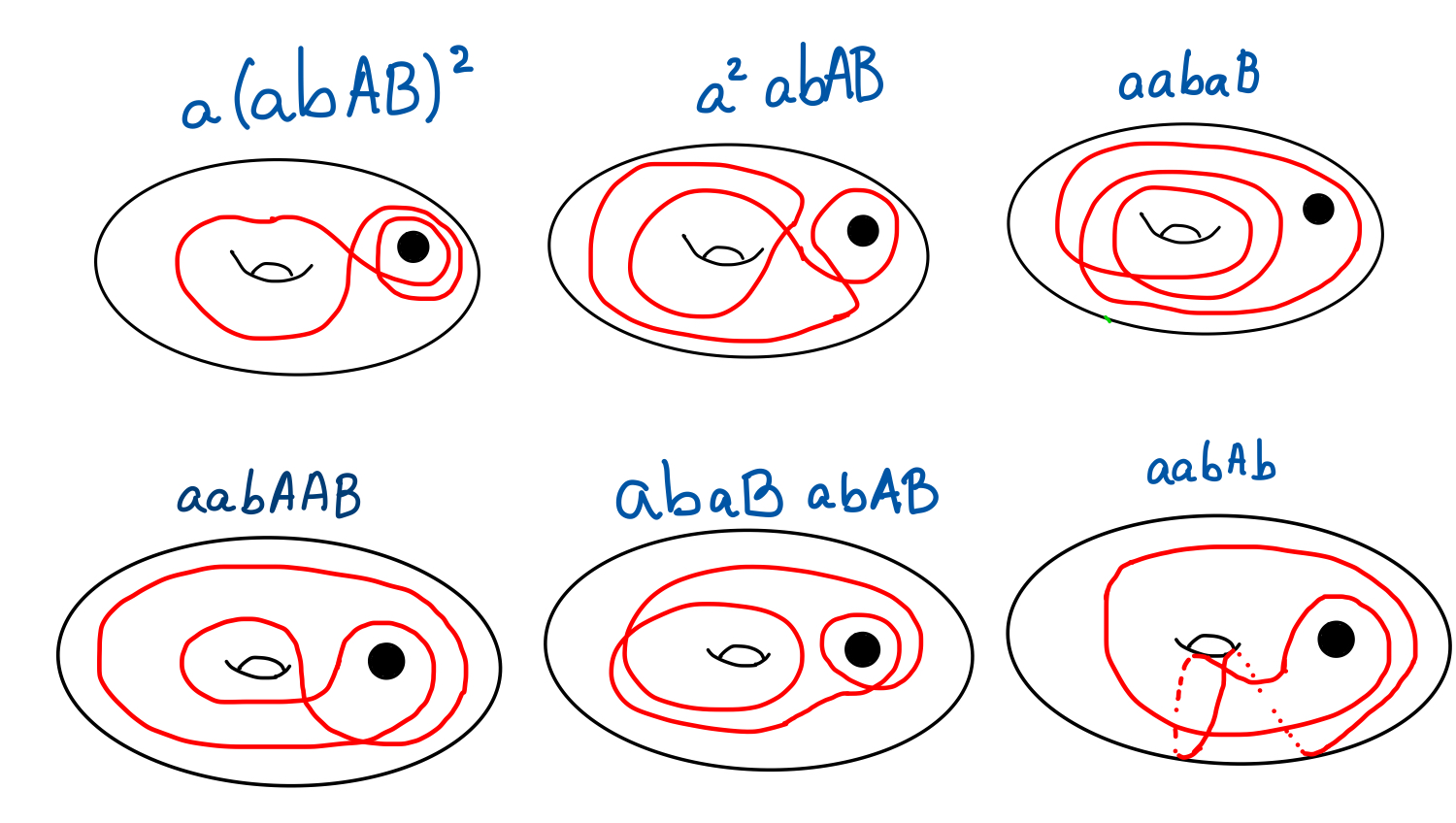} 
\caption{A representative of each orbit of self-intersection number two}\label{two2}
\end{figure}

{\renewcommand{\arraystretch}{2}
\begin{table}
\begin{tabular}{|c|c|c|c|c|c|}
\hline
$\alpha$ &$p_\alpha$. & SI& WLl of curves & $\countequal{\ell}(\alpha)$&Special\\
\hline
\hline

\text{a(abAB)$^3$} & 2 & 3 & All & $4\Phi(\ell-12)$, $\ell>12$ &-\\ \hline
\text{aabAbaBAb} & 4 &3 & All & $4\Phi(\ell-4)+4\Phi(\ell-8)$,  if $\ell>9$ &$\countequal{9}(\alpha)=16$\\ \hline
\text{aabbAB} & 4 &3 & All &$8\Phi(\ell-4)$, if $\ell>6$ &$\countequal{6}(\alpha)=4$\\ \hline
\hline
\text{a$^2$(abAB)$^2$} &  1/2 &3 &Even&$4\Phi(\ell/2-4)$,  if $\ell>9$&-\\ \hline

\text{aabAABabAB} & 1/2 &3 &Even&$4\Phi(\ell/2-4)$,  if $\ell>9$&-\\ \hline

\text{abaBAbaBAbaB} & 1/2 &3 &Even&$4\Phi(\ell/2-4)$, if $\ell>9$ &$\countequal{12}(\alpha)=8$ \\ \hline 

\text{abaBAbAB} & 1/4 &3 &Even & $2\Phi(\ell/2-4)$, if $\ell>13$&$\countequal{8}(\alpha)=4$ \\ \hline
\text{a$^4$ b$^2$} & 1 &3 &Even & $4\Phi(\ell/2)+4\Phi(\ell/2-2)$, if $\ell > 7$& $\countequal{6}(\alpha)=10$  \\ \hline
\hline\hline

\text{a$^3$abAB} &  2/9 &3 &$\cong 1 \pmod 3$&4$\Phi((\ell+2)/3-2)$ if $\ell>6$ & -\\ \hline
\text{aaabAAB} &  2/9 &3 &$\cong 1 \pmod 3$&4$\Phi((\ell+2)/3-2)$ if $\ell>6$ & -\\ \hline
\text{aabaBAbaB} &  2/9 &3 &$\cong 1 \pmod 3$&4$\Phi((\ell+2)/3-2)$ if $\ell>9$ & 	 $\countequal{9}(\alpha)=4$\\ 
\hline
\text{aabaBabAB} &  4/9 &3 &$\cong 1 \pmod 3$&8$\Phi((\ell+2)/3-2)$ if $\ell>9$ & $\countequal{9}(\alpha)=8$\\ \hline\hline

\text{aaabaB} & 1/8 &3 &Mult. of  4&$4\Phi(\ell/4)$ if $\ell>7$&$\countequal{6}(\alpha)=4$\\ \hline
\text{aabaaB} & 1/16 &3 &Mult. of  4&$2\Phi(\ell/4)$ if $\ell>7$&$\countequal{6}(\alpha)=2$\\ \hline

\end{tabular} 
\caption{The number of elements of word length $\ell$ in the orbit of curves of self-intersection number three.  
If nothing is said about a word length then the number of elements is zero. 
The number of elements up to word length $\ell$ in the orbit of $\alpha$,  $\countupto{\ell}(\alpha)$ grows like $\frac{6}{\pi^2}p_\alpha$, where $p_\alpha$ is a rational number that depends on $\alpha$ (second column)}\label{t3}
\end{table}
}

\section{Experimental study of the growth of orbits of free homotopy classes in terms of  geometric length}
\subsection{The metrics}\label{metrics}

Construction learned from Feng Luo, based on information of Bernard Maskit.

Fix a  hyperbolic metric $X$ on $T_{1,1}$, so that the boundary component is a geodesic. We can parametrize such a metric by three positive real numbers $(l_1,l_2,l_3)$ if  there is a  hyperbolic pentagon, with two consecutive sides of length $l_1$ and $l_2$, a side of length $l_3$ opposite to the angle formed by sides of length $l_1$ and $l_2$ and  right angles, except possibly for the angle of the sides of lengths $l_1$ and $l_2$. The latter angle will be acute here.

Denote by $G$ the vertex of sides with length $l_1$ and $l_2$. We are going to form a right angled octagon in the following way:  Consider the union of the pentagon and the same pentagon rotated $\pi$   about $G$. Extend the sides consecutive to those of length $l_3$ in both pentagons. Two other pentagons are determined. (See Figure~\ref{me})

Let $Y$ and $O$ denote the midpoints of the sides of the octagons perpendicular to the segments   of  sides of lengths $l_1$ and $l_2$  (in yellow and orange in Figure~\ref{me}).
Denote by $r_Y$, $r_O$ and $r_G$ the $\pi$ rotation about $Y$ and $O$ and $G$ respectively.

Set $a$ as $r_G r_Y$ and $b$ as $r_G r_O$.  This determines  a representation of our free group with generators $a$ and $b$, using the geometry.

\begin{figure}
\includegraphics[scale=.2]{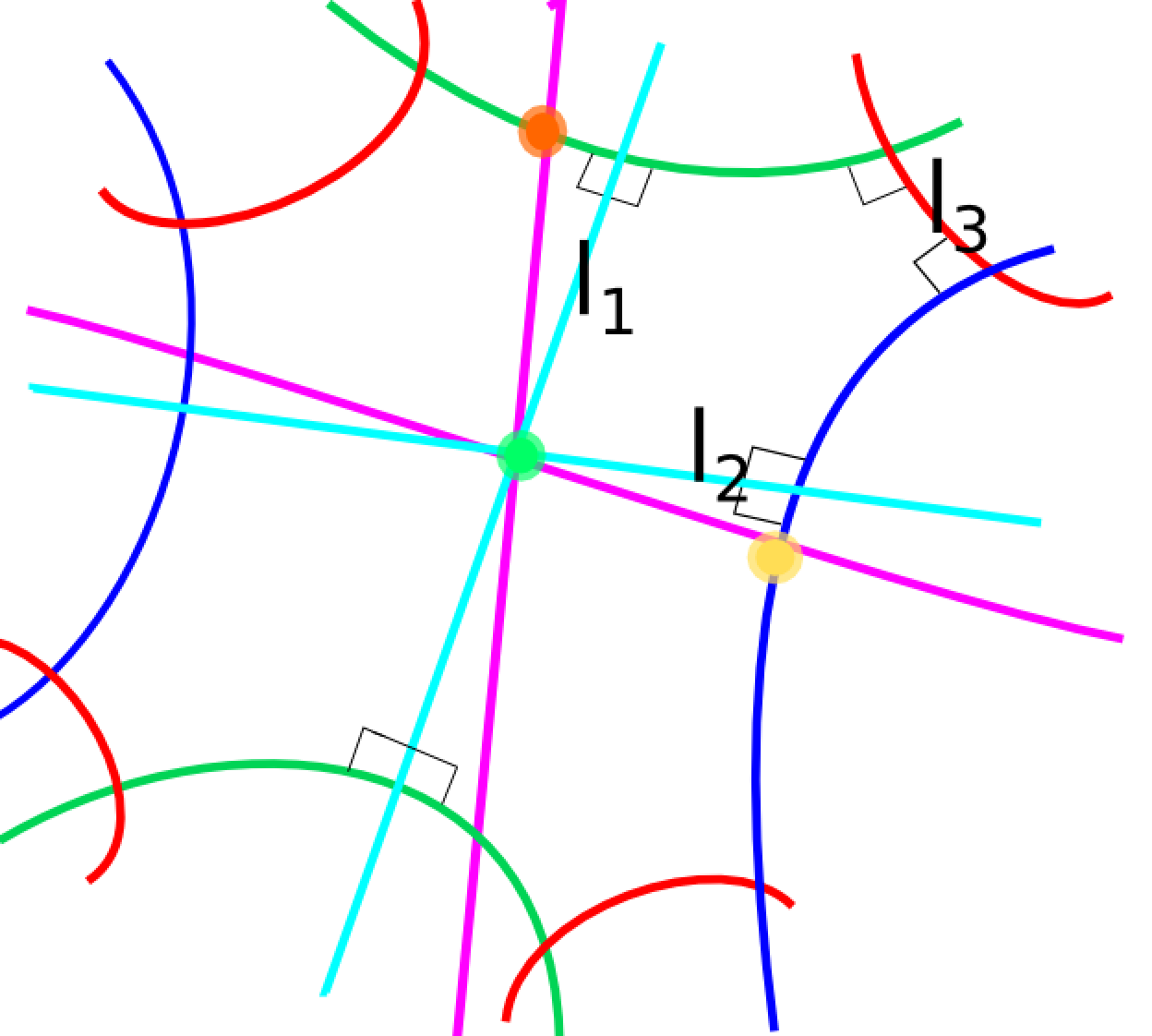} 
\caption{Metrics. The green, orange and yellow points are the midpoints. The axes of the generators are colored pink}\label{me}
\end{figure}

\textbf{How do we know we are counting all geodesics up to a certain length?}: The octagon is made of two pairs of congruent pentagons. The angle at $V$ of the original pentagon pair is chosen to be smaller than $\pi/2$. Hence the angle at $V$ of the other pentagon pair is larger than $\pi/2$. Thus  the side opposite to $V$ of the new pentagon is longer than the side opposite to $V$ of the original pentagon pentagon. This implies that any segment between "good" sides of the octagon is longer than $c=\min(2l_1,2l_2,l_3)$. Then for any geodesic $w$, 
$\gl(w) \ge c \wl(w)$.

This implies  the  inclusion.
\begin{proposition}\label{inclusion} For $c=\min(2l_1,2l_2,l_3)$,  
$$
\{w \mbox{ free hom. class/  } \gl(w)\le \ell \} \subset \{w \mbox{  free hom. class /} \wl(w)\le  \ell /c\}
$$
where the geometric length $\gl$ is the metric of parameters $(l_1,l_2,l_3)$
\end{proposition}

\subsection{The count of elements in the orbits}\label{counts}

 By \cite{m16} (see also \cite{es}) if $\gamma$  is a  geodesic in $T_{1,1}$  with respect to a hyperbolic metric $m$  then there exists a rational constant $c_\gamma$ such that the number $s_m(\ell, \gamma)$  of geodesics in their mapping class group orbit of $\gamma$, up  hyperbolic length $\ell$  grows like $ \frac{ \mu_{Th}(B_m(1))}{\pi^2/6}c_\gamma \ell^2 $. In symbols,
\begin{equation}\label{mirz}
s_m(\ell,\gamma) \sim \frac{ \mu_{Th}(B_m(1))}{\pi^2/6} c_\gamma \ell^2
\end{equation}
where  $B_m(1)$ is the unit ball in the space of measured geodesic laminations with respect to the length function at $m$ and $\mu_{Th}$ is the Thurston's measure of this ball.  Observe that $ \mu_{Th}(B_m(1))$ is a constant that depends only on the metric $m$ and not on $\gamma$. 

If $a$ is a simple curve in the punctured torus, not parallel to the boundary component, then  $c_a=\frac{1}{2!}$  (see \cite[page 110, Example 1]{mir}).
Hence 
$$
s_m(\ell,a) \sim  \frac{3}{\pi^2} \mu_{Th}(B_m(1))\ell^2.
$$

On the other hand, recall that by Proposition~\ref{simple},  $C_\ell(a)$ grows like $(6/\pi^2)\ell^2$. The similarity of the formulae for the geometric and combinatorial case prompted us to study the ratios $s_m(\ell,\gamma)/s_m(\ell,a)$ for $\ell$  for certain metrics $m$. In order to study these ratios, we estimated the coefficients of $\ell^2$ in Equation~(\ref{mirz}) in the next Subsection.

\subsection{Mirzakhani's function and its inverse}

 Using Nielsen's elementary transformations \cite{MKS}, we found all the elements of word length at most $170$ in the orbit of $\gamma$, that is, the set $\orb{170}(\gamma)$. (There are $\sum_{\ell=1}^{170} 2\phi(\ell)=17660$ classes in the orbit of the simple curve $a$. The number of elements in the orbits we studied, is of that order).

 We chose four  hyperbolic metrics with geodesic boundary  $m_1,m_2,m_3,m_4$. The parameters, $(l_1,l_2,l_3)$ of each of these four metrics are listed in the first three columns of Table~\ref{ratios}.
 
For each orbit $\orb{}(\gamma)$ of curves of self-intersection number less than or equal to three, and the four metrics $m_1,m_2,m_3,m_4$,  we computed all the $m_j$-lengths  of the geodesics in $\orb{170}(\gamma)$  and selected those smaller than 
 a given upper bound (for each metric, for computational reasons, we had to choose a different upper bound), as follows,

\begin{enumerate} 
 \item  For  each metric $m_j$ we determined the constant $c_j=\min(2l_1(j),2l_2(j),l_3(j))$ (so the inclusion Proposition~\ref{inclusion} holds) and set  $\ell_j=100/c_j$.
\item We computed (or estimated)   $$L( \ell_j,\gamma)=\{m_j(\alpha) \mbox{ : } \alpha \in \orb{\ell_j}(\gamma)\},$$ where $m_j(\alpha)$ denotes the $m_j$-length of the $\alpha$. (Note that repetitions of lengths must be allowed). 
\item For each $\gamma$ and $m_j$, we estimated the coefficient  of $\ell^2$, (that is, $\frac{3}{\pi^2} \mu_{Th}(B_m(1))$) in Equation~(\ref{mirz}), as 
$$h(\gamma, m_j)=\frac{M-u}{\sqrt{T}},$$ where 
$M=\max L( \ell_j,\gamma)$, 
$u=\min L( \ell_j,\gamma)$ and $T= \mbox{Cardinality of } L( \ell_j,\gamma)$.
\item We estimated the ratios $$h(\gamma, m_j)/h(a, m_j),$$ for all $\gamma$ of self-intersection three or less. (see Table~\ref{ratios})
\end{enumerate}

\begin{table}

\includegraphics[scale=.5]{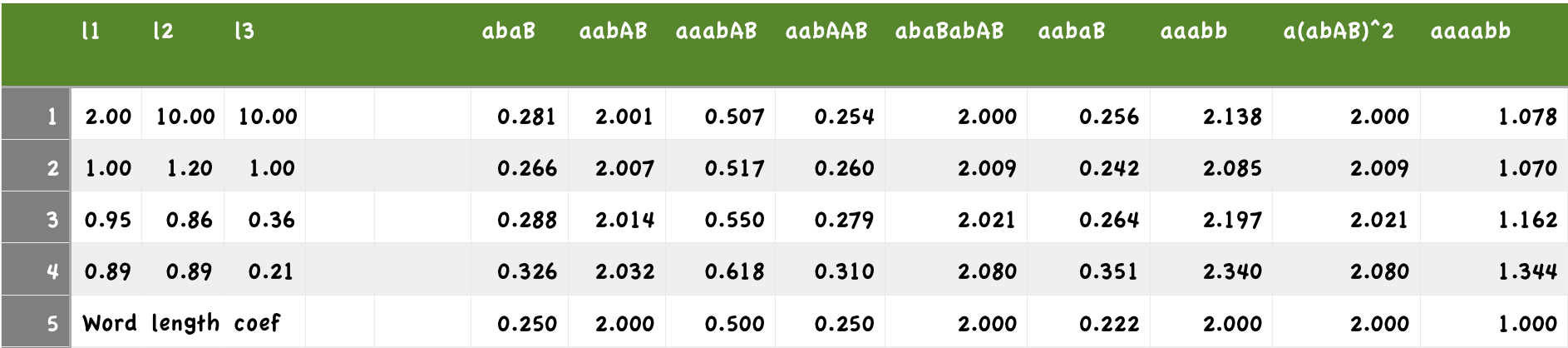} 
\caption{Ratios discuss in Subsection~\ref{counts}, (4). Also, in the last row, there is the growth coefficient of growth by word length comparison. }\label{ratios}
\end{table}

We produced the  graphs shown Figure~\ref{aabb}, explained below.  

Denote by  $\min(\gamma, m_j)$  the length of the shortest (resp. longest) $m_j$-geodesic in the orbit of $\gamma$, that is the minimum of $L(\gamma, \ell_j)$, and by  $\max(\gamma, m_j)$ the maximum of $L(\gamma, \ell_j)$. Note that $max(\gamma, m_j) \le 170/c_j$, (in general, it is  quite close).

Fix a curve $\gamma$ of self-intersection number three or less, and one of the metrics $m_j$.
To make the notation lighter, in the following lines,  set $M=max(\gamma, m_j)$, $u=min(\gamma, m_j)$, $T=s_{m_j}(\ell_j,\gamma)$.

For each of the metrics $m_1, m_2, m_3, m_4$, and each orbit of $\gamma$ of self-intersection number three or less, we produced the following graphics. In cases (i) and (ii) we study the actual graph and our estimation of this graph.

\begin{enumerate}[(i)]
\item The graph of Mirzakhani function from the reals to the integers, $\ell \mapsto s_m(\ell,\gamma)$, and the graph of the function from  the reals to  the reals $\ell \mapsto d (\ell - u)^2$, where $d=T/(M-u)^2$ is found by solving the equation $T=d \dot (M-u)^2$.
\item  The  graph of the function from the positive integers to the reals, $k \mapsto \mathrm{length}(m_j,\gamma)$ where $\mathrm{length}(m_j,\gamma))$ is the length $k$-th $m_j$-geodesic, ordered by geometric length. (``inverse'' of (1), and the graph of the function from  the reals to the reals $k \mapsto b \sqrt{k}+u$, where  $b=(M-u)/\sqrt{T}$.
\item The graph of the function $k\mapsto \frac{1}{\sqrt{k}}(-u+\mbox{length of the $k$-th geodesic in }\orb{}(\gamma))$.
\end{enumerate}

Some examples of these graphics are displayed in  Figures~\ref{aabb} and \ref{aabb1}. The even rows show the graphics up to the largest domain we could, and the odd rows show the graphics for small initial segment of this domain.

\begin{figure}
\includegraphics[scale=.325]{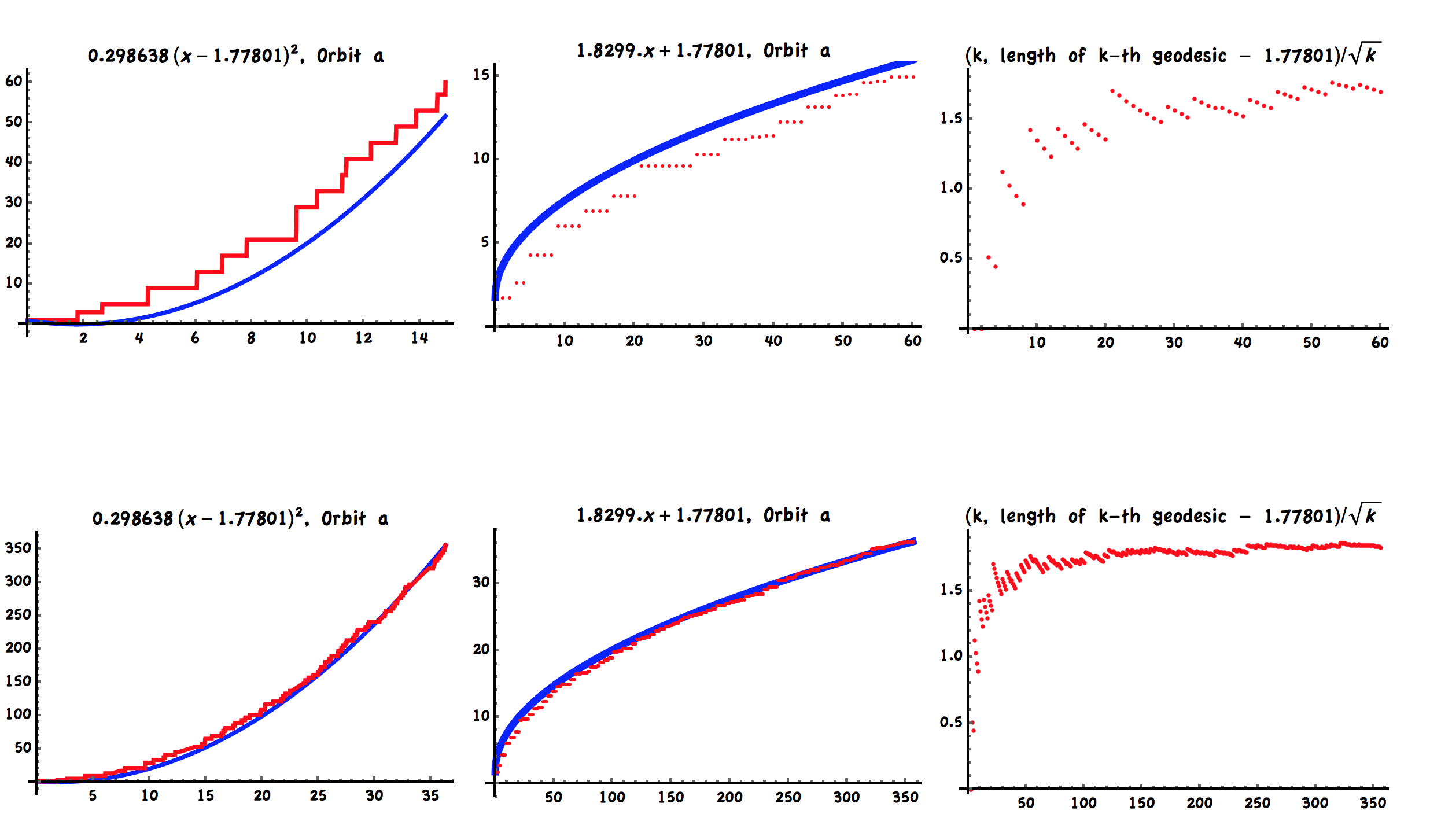} 
\includegraphics[scale=.325]{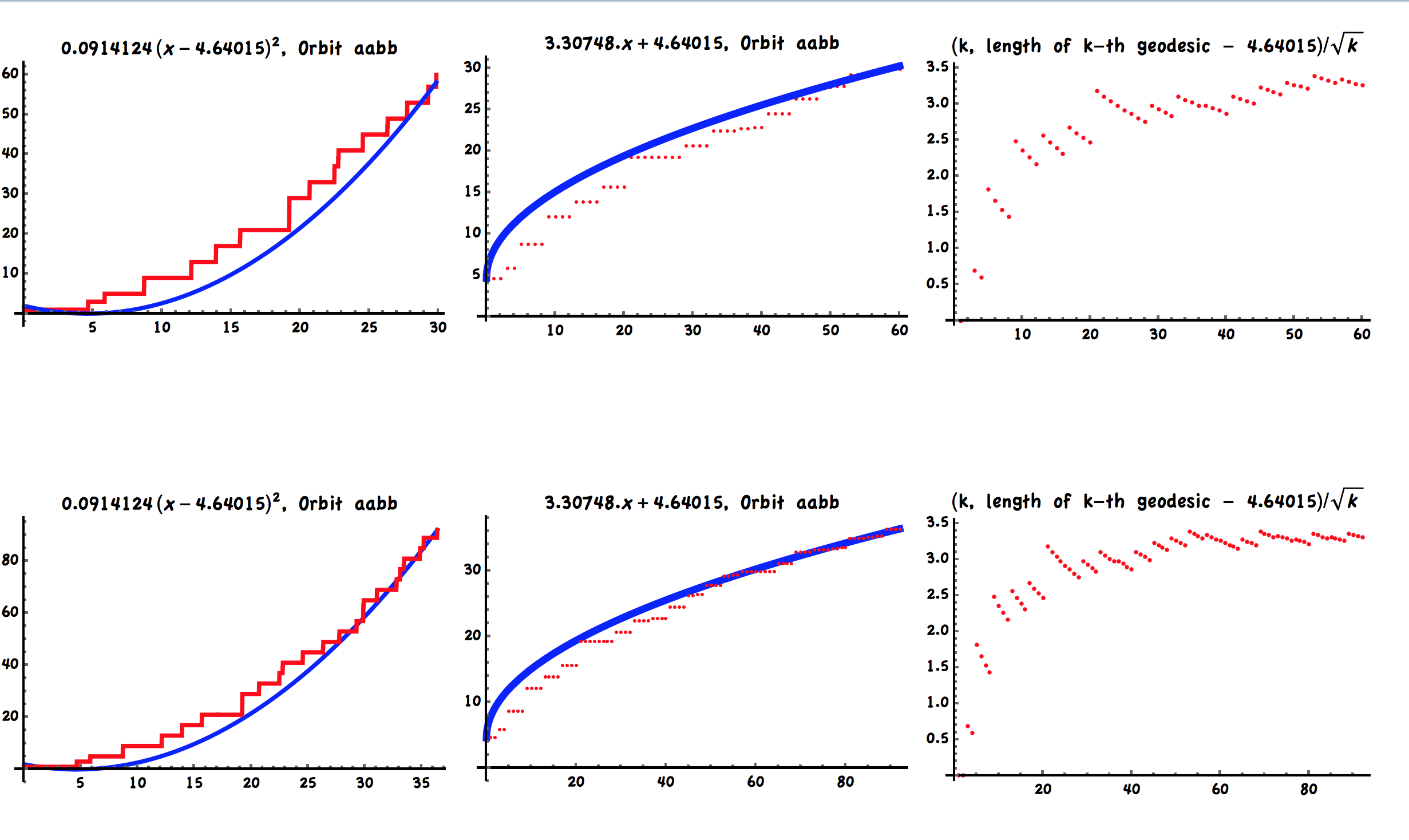} 
\includegraphics[scale=.325]{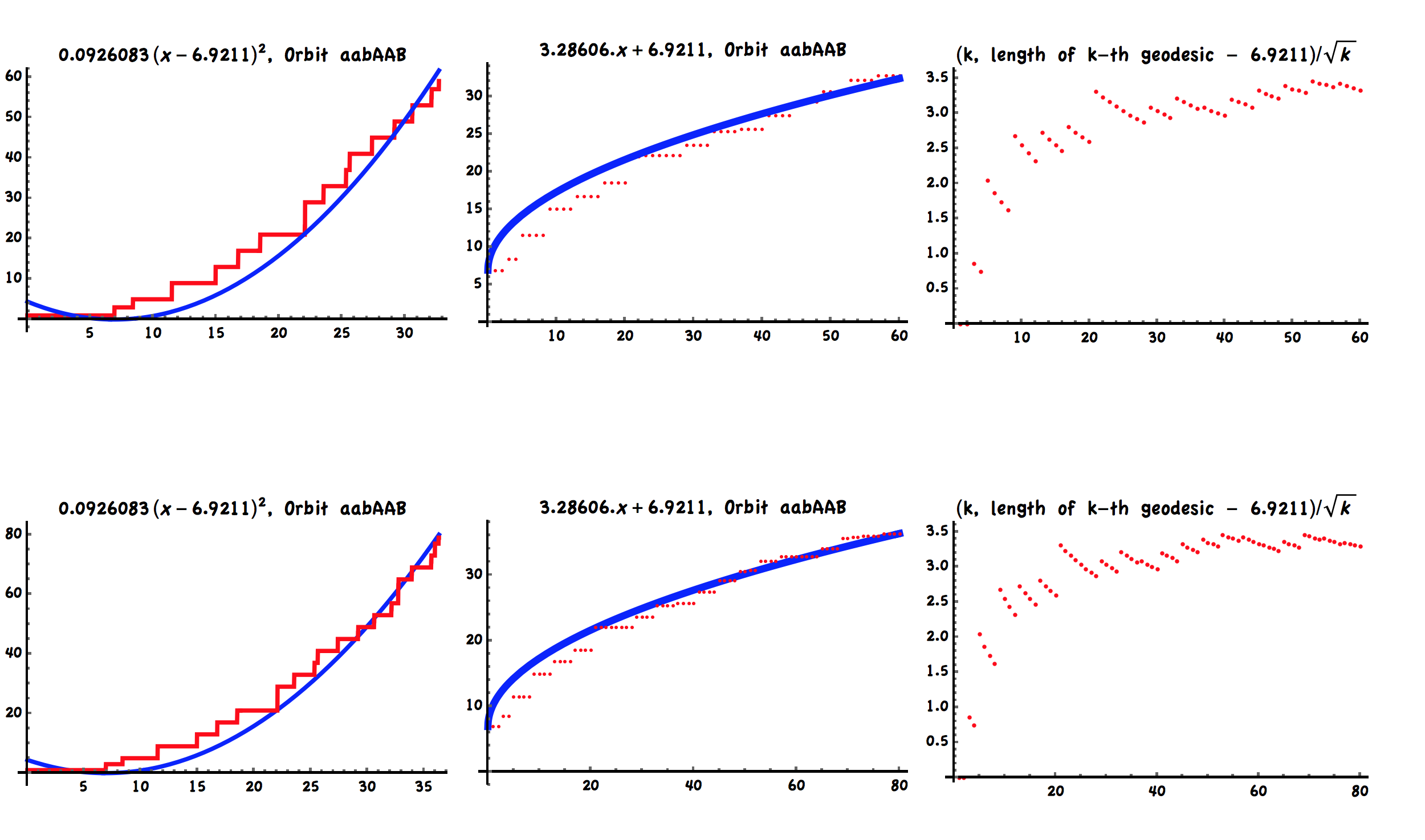} 

\caption{Study of the orbits of the curves $a$, $abaB$ and $aabAAB$ for the metric $(l_1,l_2,l_3)=(0.89,0.889,0.2149)$
}\label{aabb}
\end{figure}
\begin{figure}
\includegraphics[scale=.325]{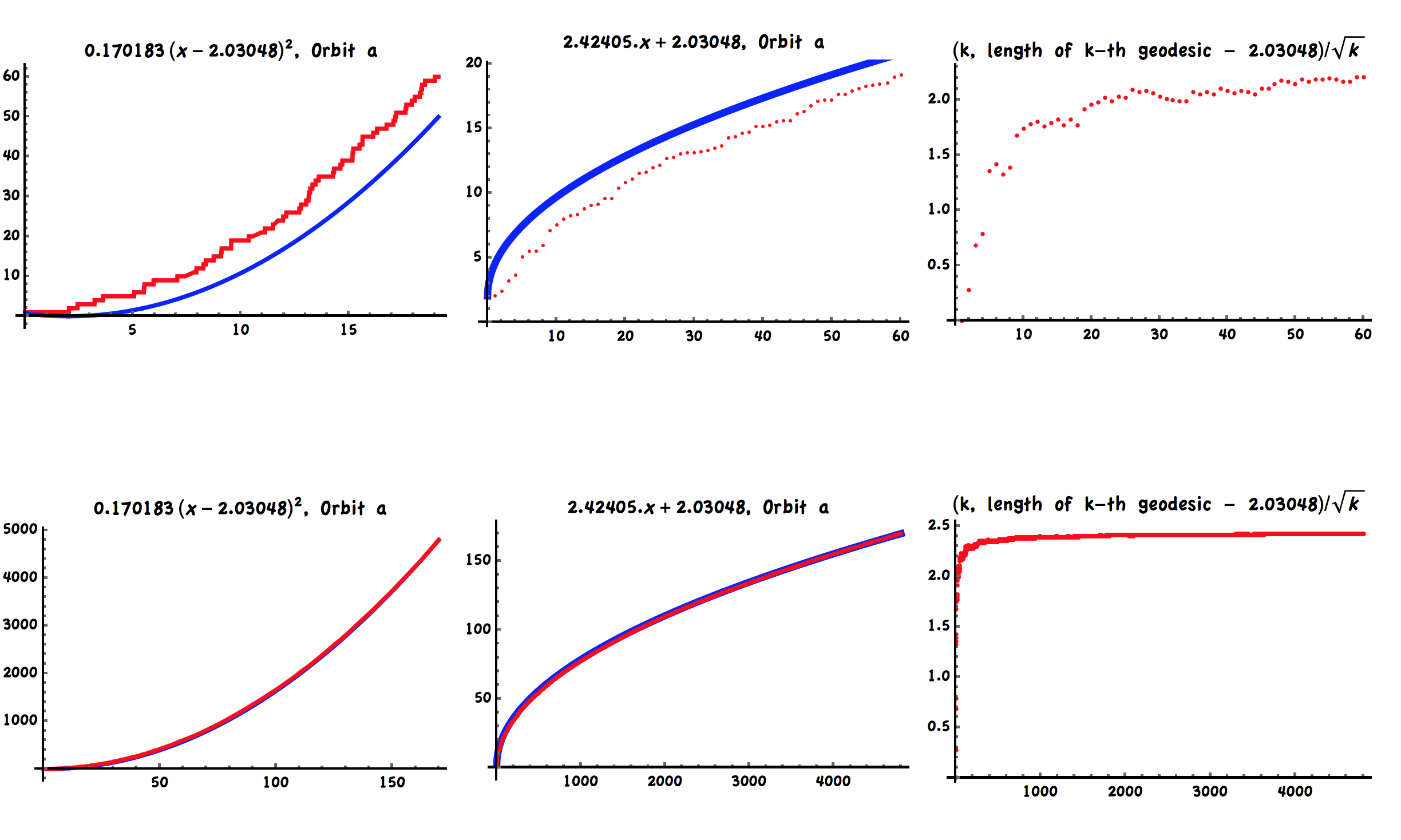} 
\includegraphics[scale=.325]{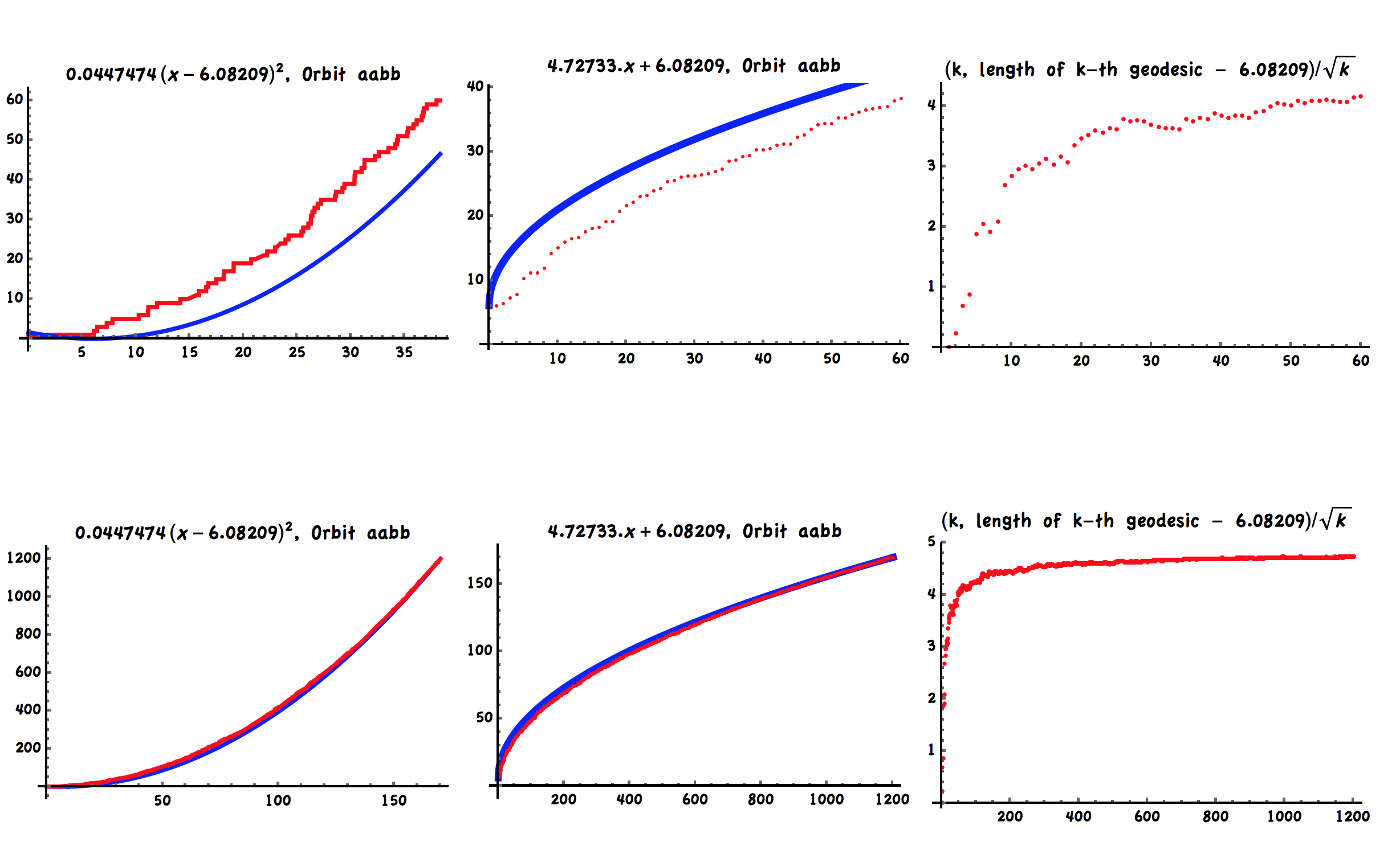} 
\includegraphics[scale=.325]{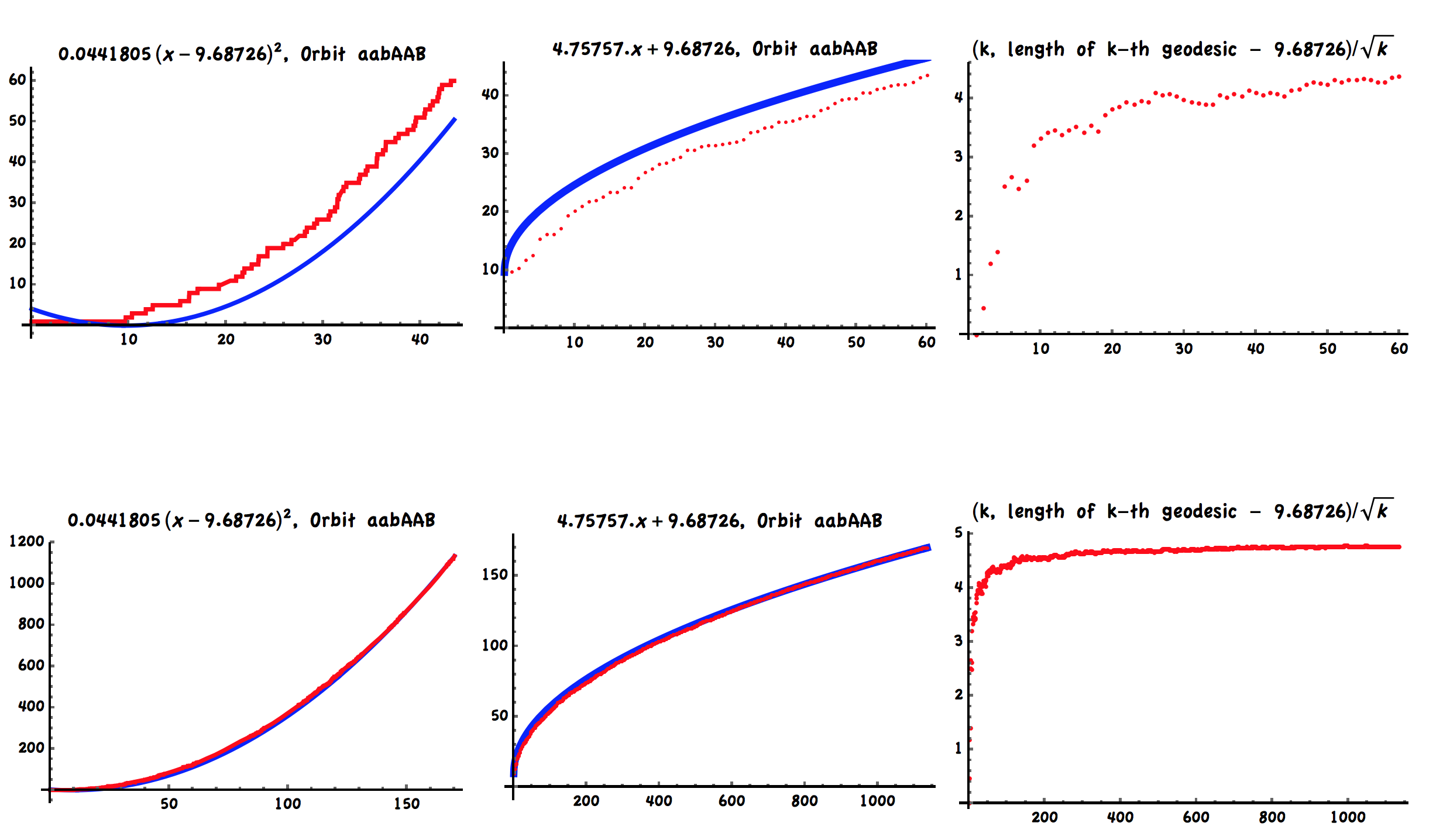}

\caption{Study of the orbits of the curves $a$, $abaB$ and $aabAAB$ for the metric $(l_1,l_2,l_3)=(1,1.2,1.012)$}\label{aabb1}
\end{figure}

\section{Conjectures}
\subsection{Geometric conjectures}

All our computations (for all classes of curves of self intersection number at most 3) suggest the following conjecture:

\begin{conjecture}\label{palpha} $p_\alpha =2c_\alpha $, where $c_\alpha$ is the "Mirzakhani" geometric length coefficient and $p_\alpha$ is the word length growth coefficient. Thus,
$$
s_m(\ell,\gamma) \sim  \frac{3}{2\pi^2} \mu_{Th}(B_m(1))p_\gamma \ell^2
$$
\end{conjecture}

\begin{conjecture}\label{inverse}  For each orbit $\orb{}(\alpha)$ and each hyperbolic metric  there exist coefficients $u,v$ such that the difference $(u\sqrt{k}+v) -\mbox{length of the $k$-th geodesic in the orbit of $\alpha$}$ is bounded by a constant that depends on $\alpha$ and the metric.
\end{conjecture}

The pictures suggest:

\begin{conjecture}\label{bounded distance}  For each orbit $\orb{}(\alpha)$ and each hyperbolic metric  there exist a quadratic polynomial $p_\alpha(\ell)=w(\ell-min(\alpha,m))^2$, (where $min(\alpha,m)$ is the length of the shortest geodesic on the orbit ) such that for all $\ell$ the difference between $p_\alpha(\ell)$ and the number of  elements in the orbit of $\alpha$ of word length up to $\ell$ is bounded.
\end{conjecture}

\subsection{Topological-combinatorial conjectures}

\begin{conjecture} \label{lc}
For each free homotopy class $\alpha$ there exist integers $j_1,\dots, j_n$,  and positive integers $k_1,\dots, k_n$, such that the number of elements of word length $\ell$ in the orbit of $\alpha$ is
$2\left(\Phi\left(\frac{\ell + j_ 1}{k_1}\right)+\Phi\left(\frac{\ell + j_ 2}{k_2}\right)+\cdots+\Phi\left(\frac{\ell + j_ n}{k_n}\right)\right)$.
In symbols,
$$
\countequal{\ell}(\alpha)=2\left(\Phi\left(\frac{\ell + j_ 1}{k_1}\right)+\Phi\left(\frac{\ell + j_ 2}{k_2}\right)+\cdots+\Phi\left(\frac{\ell + j_ n}{k_n}\right)\right).
$$
Hence the number of elements of word length less than or equal to $\ell$ in the orbit of $\alpha$ is asymptotic to $d \cdot \ell^2$, where  
 $d=\frac{6}{\pi^2}
 \left(\frac{1}{k_1^2}+\frac{1}{k_2^2}+\cdots+\frac{1}{k_n^2}\right) = 
  \left(\frac{1}{k_1^2}+\frac{1}{k_2^2}+\cdots+\frac{1}{k_n^2}\right)/
 \left(\frac{1}{1^2}+\frac{1}{2^2}+\cdots+\frac{1}{n^2}\cdots\right)
$

\end{conjecture}

\begin{remark} Conjecture \ref{lc} implies that the number of elements up to word length $\ell$  in an orbit is asymptotic to $\frac{6}{\pi^2} p \ell^2$ where $p$ is a rational number that depends on the orbit.

\end{remark}
In the cases of self-intersection number $0, 1, 2,$ and $3$ the number of orbits is smaller than or equal to the sum of the growth coefficients of all orbits and 
 surprisingly close to the sum of these growth coefficients (see Table~\ref{coefficients}). This motivates the following:

\begin{conjecture}\label{approx}
For each $k \ge 0$,  the sum of  all coefficients $p_\alpha$ over all orbits of curves of self-intersection number $k$  is within one of the total number of orbits of self-intersection number $k$.
\end{conjecture}

\begin{conjecture}\label{2and14}
\begin{enumerate}
\item The number of elements of word length $\ell$ in each of the six orbits of curves of self-intersection number two is listed in Table~\ref{t2}. 
\item The number of elements of word length $\ell$ in each of the fourteen orbits of curves of self-intersection number three is listed in  Table~\ref{t3}.
\end{enumerate}
\end{conjecture}

{\renewcommand{\arraystretch}{1.5}
\begin{table}

\begin{tabular}{|c|c|c|c|c|}
\hline 
SI&0 & 1 & 2 & 3 \\ 
\hline 
p&1 & $\frac{9}{4}$ & $\frac{197}{36}$ & $\frac{2059}{144}$ \\ 
p Approx.&1 & $2.25$ & $5.47222$ & $14.2986$ \\ 
Number of orbits&1 & $2$ & $6$ & $14$ \\ 
\hline 
\end{tabular} 
\caption{The growth of all classes of    self-intersection number $0,1,2,3$ is $p\frac{6}{\pi^2}\ell^2$  }\label{coefficients}
\end{table}
}

Conjectures \ref{palpha}, \ref{inverse}, \ref{bounded distance}, \ref{lc} and \ref{approx} can of course be generalized to all negatively curved surfaces.

\bibliographystyle{plain}
\bibliography{curves}

\enddocument